\documentclass[12pt]{article}
\usepackage[centertags]{amsmath}
\usepackage{amsfonts}
\usepackage{amssymb}
\usepackage{latexsym}
\usepackage{amsthm}
\usepackage{newlfont}
\usepackage{graphicx}
\usepackage{listings}
\usepackage{booktabs}
\usepackage{abstract}
\newlength{\defbaselineskip}
\newcommand{\setlinespacing}[1]%
           {\setlength{\baselineskip}{#1 \defbaselineskip}}

\newcommand{\N}{{\mathbb{N}}}

\newcommand{\actaqed}{\hfill $\actabox$}
{\medskip\noindent \textit{Proof of #1. }}%
{\actaqed \medskip}

\def\D{{\mathcal D}}

\def\K{{\mathcal K}}
\def\L{{\mathcal L}}
\def\R{{\mathbb R}}
\def \<{\langle}
\def\>{\rangle}

\def \e{\epsilon}

\def \sp{\operatorname{span}}

\def\la{\lambda}

\newtheorem{Theorem}{Theorem}[section]
\newtheorem{Lemma}{Lemma}[section]

\newtheorem{Remark}{Remark}[section]

\newtheorem{Corollary}{Corollary}[section]
\numberwithin{equation}{section}

\begin{document}
\title{{An inequality for the entropy numbers and its application} }
\author{V.N. Temlyakov \thanks{ University of South Carolina and Steklov Institute of Mathematics. Research was supported by NSF grant DMS-1160841}} \maketitle
\begin{abstract}
{We prove an inequality for the entropy numbers in terms of nonlinear Kolmogorov's widths. This inequality is in a spirit of known inequalities of this type and it is adjusted to the form convenient in applications for $m$-term approximations with respect to a given system. Also, we obtain upper bounds for the $m$-term approximation by the Weak Relaxed Greedy Algorithm with respect to a system which is not a dictionary.  }
\end{abstract}

\section{Introduction}

This paper was motivated by the very recent paper \cite{GIY}. The authors of \cite{GIY} study the entropy and best $m$-term approximation of the $\ell_q$-hulls of finite systems of elements in the $L_p$ spaces. They conduct this study by probabilistic methods. In this context probabilistic methods were used in some earlier papers, for instance, in \cite{DGDS}. Here we demonstrate how known results from greedy approximation in Banach spaces combined with known technique of general inequalities for the entropy numbers allow us to obtain similar results. Moreover, we show that the use of a greedy algorithm allows us to provide a deterministic construction of good $m$-term approximants. 

A number of different widths are being studied in approximation theory: 
Kolmogorov widths, linear widths, Fourier widths, Gel'fand widths, Alexandrov widths
and others. All these widths were introduced in approximation theory 
as characteristics of function classes (more generally compact sets) which give
 the best possible accuracy of algorithms with certain restrictions.
For instance, Kolmogorov's $n$-width for centrally symmetric compact set
$F$ in a Banach space $X$ is defined as follows
$$
d_n(F,X) := \inf_L \sup_{f\in F} \inf_{g\in L} \|f-g\|_X
$$
where $ \inf_L$ is taken over all $n$-dimensional subspaces of $X$.
In other words the Kolmogorov $n$-width gives the best possible error in
approximating a compact set $F$ by $n$-dimensional linear subspaces.

There has been an increasing interest last decades in nonlinear $m$-term
approximation with regard to different systems.  In \cite{T1} we  generalized the concept of classical Kolmogorov's width in order to  
use it in estimating best $m$-term approximation. For this purpose we introduced 
a nonlinear Kolmogorov's $(N,m)$-width:
$$
d_m(F,X,N) :=  \inf_{\L_N, \#\L_N \le N} \sup_{f\in F} \inf_{L\in \L_N}
\inf_{g\in L} \|f-g\|_X,
$$
where $\L_N$ is a set of at most $N$ $m$-dimensional subspaces $L$. It is clear that
$$
d_m(F,X,1) = d_m(F,X).
$$
The new feature of $d_m(F,X,N)$ is  that we allow to choose a subspace 
$L\in \L_N$ depending on $f\in F$. It is clear that the bigger $N$ the more flexibility we have 
to approximate $f$. It turns out that from the point of view of our applications the two cases 
\begin{equation}\label{1.1}
N \asymp K^m,  
\end{equation}
where $K >1 $ is a constant, and
\begin{equation}\label{1.2}
N \asymp m^{am}, 
\end{equation}
where $a >0$ is a fixed number, play an important role.

It is known (see \cite{T18}) that the $(N,m)$-widths can be used for estimating from below the best $m$-term
approximations. Let $X$ be a Banach space and let $B_X$ denote the unit ball of $X$ with the center at $0$. Denote by $B_X(y,r)$ a ball with center $y$ and radius $r$: $\{x\in X:\|x-y\|\le r\}$. For a compact set $A$ and a positive number $\e$ we define the covering number $N_\e(A,X)$
 as follows
$$
N_\e(A,X) :=\min \{n : \exists y^1,\dots,y^n :A\subseteq \cup_{j=1}^n B_X(y^j,\e)\}.
$$
It is convenient to consider along with the entropy $H_\e(A,X):= \log N_\e(A,X)$ (here and later $\log:=\log_2$) the entropy numbers $\e_k(A,X)$:
$$
\e_k(A,X) :=  \inf \{\e : \exists y^1,\dots ,y^{2^k} \in X : A \subseteq \cup_{j=1}
^{2^k} B_X(y^j,\e)\}.
$$

There are several general results (see \cite{C}) which give 
lower estimates
of the Kolmogorov widths $d_n(F,X)$ in terms of the entropy numbers
$\e_k(F,X)$. The Carl's (see \cite{C})
inequality states: for any $r>0$ we have
\begin{equation}\label{1.3}
\max_{1\le k \le n} k^r \e_k(F,X) \le C(r) \max _{1\le m \le n} m^r d_{m-1}(F,X).
\end{equation}

We proved in \cite{T1} (see also \cite{Tbook}, Section 3.5) the inequality
\begin{equation}\label{1.4}
\max_{1\le k \le n} k^r \e_k(F,X) \le C(r,K) \max _{1\le m \le n} 
m^r d_{m-1}(F,X,K^m),
\end{equation}
where we denote
$$
d_0(F,X,N) := \sup_{f\in F}\|f\|_X .
$$
This inequality is a generalization of inequality (\ref{1.3}).   
We also discussed in \cite{T1} and in Section 3.5 of \cite{Tbook} the possibility of replacing
$K^m$ by $(Kn/m)^m$ in (\ref{1.4}). The corresponding remarks (Remark 2.1 in \cite{T1} and Remark 3.5 in \cite{Tbook}) should read as follows.
\begin{Remark}\label{R1.1} Examining the proof of (\ref{1.4}) one can check that the following inequality holds
$$
n^r \e_n(F,X) \le C(r,K) \max _{1\le m \le n} 
m^r d_{m-1}(F,X,(Kn/m)^m).
$$
\end{Remark} 
In Section 2 we prove an upper bound for  $ \e_k(F,X)$ for all $k\le n$.

In Section 3 we demonstrate how the general inequality from Theorem \ref{T2.1} can be used in estimating the entropy numbers of different compacts. In particular, Corollary \ref{C3.3} gives a new proof of the corresponding upper bounds from Theorem 1 in \cite{GIY}. 

In Section 4 we study the Weak Relaxed Greedy Algorithm with respect to a system which is not a dictionary. In particular, results of Section 4 provide an algorithm which gives the same upper bounds for the best $m$-term approximation as those obtained in \cite{GIY}.

\section{A general inequality}

\begin{Theorem}\label{T2.1} Let a compact $F\subset X$ and  a number $r>0$ be such that for some
$n \in \N$
$$
  d_{m-1}(F,X,(Kn/m)^m) \le m^{-r},\quad m\le n.
$$
Then for $k\le n$
$$
\e_k(F,X) \le C(r,K) \left(\frac{\log(2n/k)}{k}\right)^r.
$$
\end{Theorem}
\begin{proof} Let $X(N,m)$ denote the union of not more than $N$ subspaces $L$ with $\dim L \le m$.
Consider a collection $\K(l) := \{X((Kn2^{-s-1})^{2^{s+1}},2^{s+1})\}^l_{s=1}$, $2^{l+1}\le n$ and denote
$$
H^r(\K(l)) := \{ f\in X : \exists L_1(f),\dots,L_l(f) : L_s(f) \in X((Kn2^{-s-1})^{2^{s+1}},2^{s+1}),
$$
and $\exists t_s(f) \in L_s(f)$ such that
$$
\|t_s(f)\|_X \le 2^{-r(s-1)}, \quad s =1,\dots,l ;\quad \|f -\sum^l_{s=1}t_s(f)\|_X \le 2^{-rl}\}.
$$
 
\begin{Lemma}\label{L2.1} We have for $r>0$
$$
\e_{2^l}(H^r(\K(l)),X) \le C(r,K)2^{-rl} (\log(Kn2^{-l}))^r,\quad 2^{l+1}\le n.
$$
\end{Lemma}
\begin{proof} We use a well known result (see, for instance, \cite{Tbook}, p. 145) to estimate  $\e_n(B_X,X)$ of
 the unit ball $B_X$ in the $d$-dimensional space $X$ :
\begin{equation}\label{2.1}
\e_n(B_X,X) \le 3( 2^{-n/d}) .  
\end{equation}
Take any sequence $\{n_s\}^{l(r)}_{s=1}$ of $l(r) \le l-2$ nonnegative integers. We will specify $l(r)$ later.
Construct $\e_{n_s}$-nets consisting of $2^{n_s}$ points each for all unit balls of the spaces in $X((Kn2^{-s-1})^{2^{s+1}},2^{s+1})$. Then
 the total number of the elements $y^s_j$ in these $\e_{n_s}$-nets does not exceed
$$
M_s := (Kn2^{-s-1})^{2^{s+1}} 2^{n_s} .
$$
We now consider  the set $A$ of elements of the form
$$
y^1_{j_1} + 2^{-r}y^2_{j_2} +\dots +2^{-r(l(r)-1)}y^{l(r)}_{j_{l(r)}}, \quad j_s \in[1,M_s],\quad
s =1,\dots, l(r).
$$
The total number of these elements does not exceed 
$$
M = \prod^{l(r)}_{s=1}M_s,\quad \log M \le \sum^{l(r)}_{s=1}2^{s+1}\log(Kn2^{-s-1})+ \sum^{l(r)}_{s=1} n_s .
$$
It is easy to see that
$$
\sum^{l(r)}_{s=1}2^{s+1}\log(Kn2^{-s-1}) \le C_12^{l(r)}\log(Kn2^{-l(r)}).
$$
  We  now set
$$
n_s := [(r+1)(l-s)2^{s+1}], \quad s=1,\dots,l(r) ,
$$ 
where $[x]$ denotes the integer part of a number $x$. We choose $l(r)\le l-2$ as a maximal 
natural number satisfying
$$
\sum^{l(r)}_{s=1}n_s \le 2^{l-1}
$$
and
$$
C_12^{l(r)}\log(Kn2^{-l(r)}) \le 2^{l-1} .
$$
It is clear that
\begin{equation}\label{1}
2^{l(r)} \ge C_22^l(\log(Kn2^{-l}))^{-1}.
\end{equation}
Then we have
$$
M\le 2^{2^l}.
$$
For the error $\e(f)$ of approximation of $f \in H^r(\K(l))$ by elements of $A$ we have
$$
\e(f) 
\le 2^{-rl} + \sum^{l(r)}_{s=1}\|t_s(f) - 2^{-r(s-1)} y^s_{j_s}\|_X + 
\sum^l_{s=l(r)+1} \|t_s(f)\|_X 
$$
$$
\le C(r)2^{-rl(r)} + \sum^{l(r)}_{s=1} 2^{-r(s-1)}\e_{n_s} (B_{L_s(f)},X) 
$$
$$
\le  C(r)2^{-rl(r)} + 3\sum^{l(r)}_{s=1} 2^{-r(s-1)} 2^{-n_s/2^{s+1}} 
\le C(r) 2^{-rl(r)} .
$$
Taking into account (\ref{1}) we complete the proof of
Lemma \ref{L2.1}.
\end{proof}

We continue the proof of Theorem \ref{T2.1}. Without loss of generality assume
$$
\max_{1\le m \le n}m^r d_{m-1}(F,X,(Kn/m)^m) < 1/2.
$$
Then for $s=1,2,\dots, l$; $l \le [\log (n-1)]$ we have
$$
d_{2^s}(F,X,(Kn2^{-s})^{2^s}) < 2^{-rs-1} .
$$
This means that for each $s = 1,2,\dots,l$, there is a collection $\L_s$ of
$(Kn2^{-s})^{2^s}$ $2^s$-dimensional spaces $L^s_j, j=1,\dots, (Kn2^{-s})^{2^s}$, such that for each
 $f\in F$ there exists a subspace $L^s_{j_s}(f)$ and an approximant 
$a_s(f)\in L^s_{j_s}(f)$ such that
$$
\|f - a_s(f)\| \le 2^{-rs-1}.
$$
Consider
\begin{equation}\label{2.2}
t_s(f) := a_s(f) - a_{s-1}(f), \quad s=2,\dots,l.
\end{equation}
Then we have 
$$
t_s(f) \in L^s_{j_s}(f)\oplus L^{s-1}_{j_{s-1}}(f), \quad \dim ( L^s_{j_s}(f)\oplus L^{s-1}_{j_{s-1}}(f))
\le 2^s + 2^{s-1} < 2^{s+1}.
$$
Note that for $K$ large enough
$$
(Kn2^{-s})^{2^s}(Kn2^{-s+1})^{2^{s-1}}\le (Kn2^{-s-1})^{2^{s+1}}.
$$
Let $X((Kn2^{-s-1})^{2^{s+1}},2^{s+1})$ denote the collection of all $ L^s_{j_s}\oplus L^{s-1}_{j_{s-1}}$ 
over various $1\le j_s\le (Kn2^{-s})^{2^s};\quad 1\le j_{s-1} \le (Kn2^{-s+1})^{2^{s-1}}$. 
For $t_s(f)$ defined by (\ref{2.2}) we have
$$
\|t_s(f)\| \le 2^{-rs-1} + 2^{-r(s-1)-1} \le 2^{-r(s-1)} .
$$
Next, for $a_1(f) \in L^1(f)$ we have
$$
\|f-a_1(f)\| \le 1/2
$$
and from $d_0(F,X) \le 1/2$ we get
$$
\|a_1(f)\| \le 1.
$$
Take $t_1(f) = a_1(f)$. Then we have
$F\subset H^r(\K(l))$ and Lemma \ref{L2.1} gives the required bound 
$$
\e_{2^l} (F) \le C(r,K)2^{-rl}(\log(Kn2^{-l}))^r, \quad 1\le l \le [\log (n-1)] .
$$
It is clear that these inequalities imply the conclusion of Theorem \ref{T2.1}.
\end{proof}
 
\section{Applications}

We begin with an application which motivated a study of $d_m(F,X,N)$ with $N=(Kn/m)^m$. Let $\D=\{g_j\}_{j=1}^n$ be a system of normalized elements of cardinality $|\D|=n$ in a Banach space $X$. Consider best $m$-term approximations of $f$ with respect to $\D$
$$
\sigma_m(f,\D)_X:= \inf_{\{c_j\};\Lambda:|\Lambda|=m}\|f-\sum_{j\in \Lambda}c_jg_j\|.
$$
For a function class $F$ set
$$
\sigma_m(F,\D)_X:=\sup_{f\in F}\sigma_m(f,\D)_X.
$$
Then it is clear that for any system $\D$, $|\D|=n$,
$$
d_m(F,X,\binom{n}{m})\le \sigma_m(F,\D)_X.
$$
Next,
$$
\binom{n}{m} \le (en/m)^m.
$$
Thus Theorem \ref{T2.1} implies the following theorem.
\begin{Theorem}\label{T3.1} Let a compact $F\subset X$ be such that there exists a normalized system $\D$, $|\D|=n$, and  a number $r>0$ such that 
$$
  \sigma_m(F,\D)_X \le m^{-r},\quad m\le n.
$$
Then for $k\le n$
\begin{equation}\label{3.0}
\e_k(F,X) \le C(r) \left(\frac{\log(2n/k)}{k}\right)^r.
\end{equation}
\end{Theorem}
\begin{Remark}\label{R3.1} Suppose that a compact $F$ from Theorem \ref{T3.1} belongs to an $n$-dimensional subspace $X_n:=\sp(\D)$. Then in addition to (\ref{3.0}) we have 
 for $k\ge n$
\begin{equation}\label{3.0'}
\e_k(F,X) \le C(r)n^{-r}2^{-k/n}.
\end{equation}
\end{Remark}
\begin{proof}  
Inequality (\ref{3.0'}) follows from Theorem \ref{T3.1} with $X=X_n$, $k=n$, inequality (\ref{2.1}) and a simple well known inequality
\begin{equation}\label{3.2}
\e_{k_1+k_2}(A,X_n) \le \e_{k_1}(A,X_n)\e_{k_2}(B_{X_n},X_n),  
\end{equation}
where $A$ is a compact and $B_{X_n}$ is a unit ball of $X_n$. 
\end{proof}

As a corollary of Theorem \ref{T3.1} and Remark \ref{R3.1} we obtain the following classical bound. 
\begin{Corollary}\label{C3.1} For any $0< q\le \infty$ and $\max(1,q) \le p\le \infty$ we have
$$
 \e_k(B^n_q,\ell^n_p) \le C(q,p) \left\{\begin{array}{ll} (\frac{\log(2n/k)}{k})^{1/q-1/p}, & k\le n\\
 2^{-k/n}n^{1/p-1/q}, &  k\ge n.\end{array} \right.
$$
\end{Corollary}
\begin{proof} Indeed, it is well known and easy to check that for a sequence of nonnegative numbers $x_1\ge x_2\ge  \cdots \ge x_n$ we have for $0<q\le p$
\begin{equation}\label{3.1}
\left(\sum_{j=m+1}^n x_j^p\right)^{1/p} \le m^{\frac{1}{p}-\frac{1}{q}}\left(\sum_{j=1}^n x_j^q\right)^{1/q}.
\end{equation}
Therefore, for $0<q\le p$
$$
\sigma_m(B^n_q,\{e_j\}_{j=1}^n)_{\ell^n_p} \le m^{\frac{1}{p}-\frac{1}{q}},\quad m\le n,
$$
where $\{e_j\}_{j=1}^n$ is a canonical basis for $\R^n$.
Applying Theorem \ref{T3.1} and Remark \ref{R3.1} we obtain   Corollary \ref{C3.1}.  
\end{proof}
For a normalized system $\D$ define $A_q(\D)$, $q>0$, as a closure in $X$ of the set
$$
\{x: x=\sum_jc_jg_j, \, g_j\in\D,\,\sum_j|c_j|^q\le 1\}.
$$
\begin{Corollary}\label{C3.2} Let $1<p<\infty$. For a normalized system $\D$ of cardinality $|\D|=n$ we have
\begin{equation}\label{3.2'}
\e_k(A_1(\D),L_p)\le C(p)\left(\frac{\log(2n/k)}{k}\right)^{\max(\frac{1}{2},\frac{1}{p})-1},\quad k\le n.
\end{equation}
\end{Corollary}
\begin{proof} It is known (see \cite{DGDS} and \cite{Tbook}) that
\begin{equation}\label{3.3}
\sigma_m(A_1(\D),\D)_{L_p} \le C(p) m^{\max(\frac{1}{2},\frac{1}{p})-1}.
\end{equation}
It remains to apply Theorem \ref{T3.1}. 
\end{proof}

\begin{Corollary}\label{C3.3} Let $\D$ be a normalized system of cardinality $|\D|=n$. Then for $0<q\le 1$ and $1<p<\infty$ we have
$$
\e_k(A_q(\D),L_p)\le C(q,p) \left\{\begin{array}{ll} (\frac{\log(2n/k)}{k})^{1/q-\max(\frac{1}{2},\frac{1}{p})}, & k\le n\\
 2^{-k/n}n^{\max(\frac{1}{2},\frac{1}{p})-1/q}, &  k\ge n.\end{array} \right.
$$
\end{Corollary}
\begin{proof} We estimate $\sigma_m(A_q(\D),\D)_{L_p}$. If $q=1$ then the bound is given by (\ref{3.3}). If $q<1$ then we use (\ref{3.1}) with $p=1$ and by (\ref{3.3}) we get
$$
\sigma_{2m}(A_q(\D),\D)_{L_p} \le C(q,p)m^{\max(\frac{1}{2},\frac{1}{p})-\frac{1}{q}}.
$$
Applying Theorem \ref{T3.1} and Remark \ref{R3.1} we obtain   Corollary \ref{C3.3}.  
\end{proof}
We note that Corollary \ref{C3.3} gives the same upper bounds as in Theorem 1 of \cite{GIY}. 
It is proved in \cite{GIY} that these bounds are best possible up to a constant.

\section{A greedy algorithm}

In Section 3 we showed how best $m$-term approximations can be used for estimating 
the entropy numbers. Here we note that $m$-term approximations are very important by themselves in 
the context of sparse approximation. In this context an important problem is to provide an algorithm that 
builds a good $m$-term approximation. We discuss a 
greedy algorithm in this section. The theory of 
greedy approximation is well developed (see \cite{Tbook}). A typical problem of greedy approximation is 
a problem of $m$-term approximation with respect to a dictionary. We say that a set of elements (functions) $\D$ from a Banach space $X$ is a dictionary, respectively, symmetric dictionary, if each $g\in \D$ has norm bounded by one ($\|g\|\le1$),
$$
g\in \D \quad \text{implies} \quad -g \in \D,
$$
and the closure of $\sp \D$ is $X$. We denote the closure (in $X$) of the convex hull of $\D$ by $A_1(\D)$. In this section we discuss greedy algorithms with regard to a system 
 $\D$ that is not a dictionary. Here, we will discuss a variant of the Weak Relaxed Greedy Algorithm (WRGA).  Let $X$ be a real Banach space and let $\D:=\{g\}$ be a system of elements $g\in X$ such that $\|g\|\le 1$ and $g\in\D$ implies $-g\in\D$. Usually, in the theory of greedy algorithms we consider approximation with regard to a dictionary $\D$. One of the properties of a dictionary $\D$ is that the closure of $\sp \D$ is equal to $X$. In this section we do not assume that the system $\D$ is a dictionary. In particular, we do not assume that the closure of $\sp \D$ is $X$. This setting is motivated by applications in Learning Theory (see Chapter 4 of \cite{Tbook}). 
 
For a nonzero element $f\in X$ we let $F_f$ denote a norming (peak) functional for $f$: 
$$
\|F_f\| =1,\qquad F_f(f) =\|f\|.
$$
The existence of such a functional is guaranteed by Hahn-Banach theorem. 
 
 Let 
$\tau := \{t_k\}_{k=1}^\infty$ be a given weakness sequence of   numbers $t_k \in[0,1]$, $k=1,\dots$. 

 {\bf Weak Relaxed Greedy Algorithm (WRGA).} 
We define $f^r_0 := f^{r,\tau}_0 :=f$ and $G^r_0:=G^{r,\tau}_0 := 0$. Then, for each $m\ge 1$ we have the following inductive definition.

(1) $\varphi^{r}_m := \varphi^{r,\tau}_m \in \D$ is any element satisfying
$$
F_{f^{r}_{m-1}}(\varphi^{r}_m - G^r_{m-1}) \ge t_m \sup_{g\in \D} F_{f^{r}_{m-1}}(g - G^r_{m-1}).
$$

(2) Find $0\le \lambda_m \le 1$ such that
$$
\|f-((1-\la_m)G^r_{m-1} + \la_m\varphi^{r}_m)\| = \inf_{0\le \la\le 1}\|f-((1-\la)G^r_{m-1} + \la\varphi^{r}_m)\|
$$
and define
$$
G^r_m:= G^{r,\tau}_m := (1-\la_m)G^r_{m-1} + \la_m\varphi^{r}_m.
$$

(3) Let
$$
f^{r}_m := f^{r,\tau}_m := f-G^r_m.
$$
 
For a Banach space $X$ we define the modulus of smoothness
$$
\rho(u) := \sup_{\|x\|=\|y\|=1}(\frac{1}{2}(\|x+uy\|+\|x-uy\|)-1).
$$
The uniformly smooth Banach space is the one with the property
$$
\lim_{u\to 0}\rho(u)/u =0.
$$
The following theorem was proved in \cite{T15} (see also Theorem 6.17 on p. 348 in \cite{Tbook}) for $\D$ being a dictionary.
\begin{Theorem}\label{T4.1} Let $X$ be a uniformly smooth Banach space with modulus of smoothness $\rho(u) \le \gamma u^q$, $1<q\le 2$. Then, for a sequence $\tau := \{t_k\}_{k=1}^\infty$, $t_k \le 1$, $k=1,2,\dots,$ we have for any $f\in A_1(\D)$ that 
$$
\|f^{r,\tau}_m\| \le C_1(q,\gamma)\left(1+\sum_{k=1}^m t_k^p\right)^{-1/p},\quad p:= \frac{q}{q-1},
$$
with a constant $C_1(q,\gamma)$ which may depend only on $q$ and $\gamma$.
\end{Theorem}
We prove here an analog of the above theorem when we do not assume that $\D$ is a dictionary and only assume that $\D=\{g\}$ is a symmetric system with a property $\|g\|\le 1$.
\begin{Theorem}\label{T4.2} Let $X$ be a uniformly smooth Banach space with modulus of smoothness $\rho(u) \le \gamma u^q$, $1<q\le 2$. Then, for a sequence $\tau := \{t_k\}_{k=1}^\infty$, $t_k \le 1$, $k=1,2,\dots,$ we have for any $f\in X$ that 
$$
\|f^{r,\tau}_m\| \le \inf_{\phi\in A_1(\D)}\|f-\phi\| + C_2(q,\gamma)\left(1+\sum_{k=1}^m t_k^p\right)^{-1/p},\quad p:= \frac{q}{q-1},
$$
with a constant $C_2(q,\gamma)$ which may depend only on $q$ and $\gamma$.
\end{Theorem}
\begin{Remark}\label{R4.1} In case of a Hilbert space $H$ there are stronger results for similar greedy algorithms with $\tau=\{1\}$ (see \cite{Tbook}, p. 99, Theorem 2.28):
$$
\|f_m\|^2_H \le \left( \inf_{\phi\in A_1(\D)}\|f-\phi\|_H\right)^2 + Cm^{-1}.
$$
 \end{Remark}
\begin{proof}  Proof of Theorem \ref{T4.2} is similar to the proof of Theorem \ref{T4.1}. 
Denote
$$
b:=\inf_{\phi\in A_1(\D)}\|f-\phi\|.
$$
We use the following lemma.
\begin{Lemma}\label{L4.1} Let $X$ be a uniformly smooth Banach space with modulus of smoothness $\rho(u)$. Then, for a given $f\in A_1(\D)$   we have
$$
\|f^{r,\tau}_m\| \le  \inf_{0\le\la\le 1}\big(\|f^{r,\tau}_{m-1}\|-\la t_m(\|f^{r,\tau}_{m-1}\|-b) 
$$
$$
+ 2\|f^{r,\tau}_{m-1}\|\rho\left(\frac{2\la}{\|f^{r,\tau}_{m-1}\|}\right)\big),
\quad m=1,2,\dots .
$$
\end{Lemma}
\begin{proof} We have  
$$
f^{r}_m := f-((1-\la_m)G^r_{m-1}+\la_m\varphi^{r}_m) = f^{r}_{m-1}-\la_m(\varphi^{r}_m-G^r_{m-1})
$$
and
$$
\|f^{r}_m\| = \inf_{0\le \la\le 1}\|f^{r}_{m-1}-\la(\varphi^{r}_m-G^r_{m-1})\|.
$$
We have from the definition of the modulus of smoothness for any $\la$
 $$
\|f^{r}_{m-1}-\la(\varphi^{r}_m-G^r_{m-1})\| +\|f^{r}_{m-1}+\la(\varphi^{r}_m-G^r_{m-1})\| \le
$$
\begin{equation}\label{4.1}
2\|f^{r}_{m-1}\|(1+\rho(\frac{\la\|\varphi^{r}_m-G^r_{m-1}\|}{\|f^{r}_{m-1}\|})).
 \end{equation}
Next we get for $\la \ge 0$
$$
\|f^{r}_{m-1}+\la(\varphi^{r}_m-G^r_{m-1})\| \ge F_{f^{r}_{m-1}}(f^{r}_{m-1}+\la (\varphi^{r}_m -G^r_{m-1})) =
$$
$$
\|f^{r}_{m-1}\| +\la F_{f^{r}_{m-1}}(\varphi^{r}_m -G^r_{m-1}) \ge \|f^{r}_{m-1}\|+\la t_m \sup_{g\in \D} F_{f^{r}_{m-1}}(g-G^r_{m-1}).
$$
Using Lemma 6.10, p. 343, from \cite{Tbook} we continue
$$
= \|f^{r}_{m-1}\|+\la t_m\sup_{\phi \in A_1(\D)} F_{f^{r}_{m-1}}(\phi-G^r_{m-1})\ge \|f^{r}_{m-1}\| +\la t_m (\|f^{r}_{m-1}\|-b).
$$
Using the trivial estimate $\|\varphi^{r}_m -G^r_{m-1}\| \le 2$ we obtain  from (\ref{4.1})
$$
\|f^{r}_{m-1}-\la (\varphi^{r}_m-G^r_{m-1})\|  
$$
\begin{equation}\label{4.2}
  \le \|f^{r}_{m-1}\|-\la t_m(\|f^{r}_{m-1}\|-b) +2 \|f^{r}_{m-1}\|\rho (\frac{2\la}{\|f^{r}_{m-1}\|})),
\end{equation}
  which proves Lemma \ref{L4.1}.
\end{proof}
Set
$$
a_m:=\|f^r_m\|-b.
$$
Note that 
$$0\le a_m\le 2.$$ 
Using monotonicity of $\rho(u)/u$ we derive from Lemma \ref{L4.1}
\begin{equation}\label{4.3}
a_m\le a_{m-1}\inf_{\la\in[0,1]}(1-\la t_m + 2\rho(2\la/a_{m-1})).
\end{equation}
For $\rho(u)\le \gamma u^q$ it gives
\begin{equation}\label{4.4}
a_m\le a_{m-1}\inf_{\la\in[0,1]}(1-\la t_m + 2\gamma(2\la/a_{m-1})^q).
\end{equation}
Denote $\la_1$   the solution of the equation
$$
\frac{1}{2}\la t_m = 2\gamma \left(\frac{2\la}{a_{m-1}}\right)^q,\quad \la_1=\left(\frac{t_ma_{m-1}^q}{2^{q+2}\gamma}\right)^{\frac{1}{q-1}}.
$$
If $\la_1\le 1$ then 
$$
\inf_{\la\in[0,1]}(1-\la t_m + 2\gamma(2\la/a_{m-1})^q)\le  1-\la_1 t_m + 2\gamma(2\la_1/a_{m-1})^q)
$$
$$
=1-\frac{1}{2}\la_1 t_m =1- C_3(q,\gamma)t_m^pa_{m-1}^{p},\quad p:=\frac{q}{q-1}.
$$
If $\la_1>1$ then for all $\la\le\la_1$ we have 
$$
 \frac{1}{2}\la t_m \ge 2\gamma \left(\frac{2\la}{a_{m-1}}\right)^q.
$$
Specifying $\la=1$ we get
$$
\inf_{\la\in[0,1]}(1-\la t_m + 2\gamma(2\la/a_{m-1})^q)\le1-\frac{1}{2} t_m \le 1-C_4(q,\gamma)t_m^pa_{m-1}^{p}.
$$
Setting $C_5:=C_5(q,\gamma):= \min(C_3(q,\gamma),C_4(q,\gamma))$ we obtain
\begin{equation}\label{4.5}
a_m\le a_{m-1}(1-C_5t_m^pa_{m-1}^{p}).
\end{equation}
It is known (see \cite{Tbook}, p. 345) that inequalities (\ref{4.5}) imply
$$
a_m \le C_6(q,\gamma) \left(1+\sum_{n=1}^m t_n^p\right)^{1/p}.
$$
This completes the proof of Theorem \ref{T4.2}.
\end{proof}

It is known (see, for instance, \cite{DGDS}, Lemma B.1) that in the case $X=L_p$ we have 
$$
\rho(u) \le  u^p/p \quad\text{ if} \quad 1\le p\le 2\quad\text{ and}\quad \rho(u)\le
(p-1)u^2/2 \quad \text{if}\quad 2\le p<\infty.   
$$
Therefore, in this case Theorem \ref{T4.2} gives: for any $f\in L_p$  
\begin{equation}\label{4.6}
\|f^{r,\tau}_m\|_{L_p} \le \inf_{\phi\in A_1(\D)}\|f-\phi\|_{L_p} + C(p)\left(1+\sum_{k=1}^m t_k^s\right)^{-1/s},
\end{equation}
where $s:= \max(\frac{p}{p-1},2)$.
It was proved in \cite{GIY} that for   $0<v\le 1$,
\begin{equation}\label{4.7}
\sigma_m(f,\D)_{L_p} \le \inf_{\phi\in A_v(\D)}\|f-\phi\|_{L_p} + C(p)m^{\max(1/p,1/2)-1/v}.
\end{equation}
The proof in \cite{GIY} is probabilistic and does not provide a deterministic algorithm for constructing a good $m$-term approximation. 
We note that inequality (\ref{4.6}) shows that in case $v=1$ the greedy algorithm WRGA with $\tau=\{t\}$ provides the rate of approximation as in (\ref{4.7}).

\end{document}